\documentclass{birkjour}
 \newtheorem{thm}{Theorem}[section]
 \newtheorem{cor}[thm]{Corollary}
 \newtheorem{lem}[thm]{Lemma}
 \newtheorem{prop}[thm]{Proposition}
 \theoremstyle{definition}
 \newtheorem{defn}[thm]{Definition}
 \theoremstyle{remark}
 \newtheorem{rem}[thm]{Remark}
 \newtheorem*{ex}{Example}
 \numberwithin{equation}{section}
\usepackage{tikz-cd}
\usepackage{amsmath}
\usepackage{amssymb}
\usepackage{tikz-cd}
\usepackage[hyphens]{url} 

\renewcommand{\leq}{\leqslant}
\renewcommand{\geq}{\geqslant}
\renewcommand{\setminus}{\smallsetminus}
\usepackage{graphicx}
\allowdisplaybreaks
\usepackage[all,cmtip]{xy}

\begin{document}

%-------------------------------------------------------------------------
% editorial commands: to be inserted by the editorial office
%
%\firstpage{1} \volume{228} \Copyrightyear{2004} \DOI{003-0001}
%\seriesextraline{This is the Concrete Title of this Book\br H.E. R and S.T.C. W, Eds.}
%\firstpage{1}

\title[Ideal Spaces of the Haagerup Tensor Product of TROs]
 {Ideal Spaces of the Haagerup Tensor Product of Ternary Rings of Operators}

%----------Author 1
\author[V. Rajpal]{Vandana Rajpal}
\address{Department of Mathematics \\  Shivaji College, \\ University of Delhi, \\ Delhi-110 027, India}

\email{vandanarajpal@shivaji.du.ac.in}

%----------Author 2
\author[A. Kansal]{Arpit Kansal} 
\address{Department of Mathematics \\   Shyama Prasad Mukherji College for Women, \\ University
of Delhi, \\ Delhi-110 026}
\email{arpit@spm.du.ac.in}


\begin{abstract}
We characterize the primal, factorial, and Glimm ideals of the Haagerup tensor product $V\otimes^{h} B$ of a TRO $V$ and a $C^{\ast}$-algebra $B$.
\end{abstract}

\date{August 15, 2025}

%%% ----------------------------------------------------------------------
\maketitle
%%% ----------------------------------------------------------------------
%\tableofcontents

\section{Introduction}
The ideal structure of the Banach algebra  $A \otimes^h B$ of $C^{\ast}$-algebras $A$ and $B$ has been extensively studied,  particularly in classifying primal, Glimm, and factorial ideals in terms of those of $A$ and $B$ (\cite{AS}, \cite{AR}).  Recent work has established deep connections between the ideals of 
$V \otimes^h B$ and those of $V$ and $B$, showing how maximal, prime, and primitive ideals correspond naturally to their counterparts in $V$ and $B$ \cite{AKKKK}.

In this paper, we continue this study by providing a comprehensive characterization of the factorial, primal and Glimm ideals of $V \otimes^h B$. Our results further strengthen the structural parallels between the ideals of $V \otimes^h B$ and those of $V$ and $B$, extending the homeomorphic correspondences developed in earlier works. The structure of the paper is as follows: 

After setting up the necessary preliminaries in Section $2$, including basic properties of the ideals of TRO, the Haagerup tensor product, and the ideals of $V \otimes^h B$, we analyze the structure of the factorial, primal and Glimm ideals of a TRO
$V$ in Section $3$. We introduce the notion of quasi-standardness for a ternary ring of operators  $V$, and establish several equivalent conditions characterizing this property. There is a natural homeomorphism of $\operatorname{Id}(V)$ onto $\operatorname{Id}(\mathcal{A}(V))$ sending $I \to \mathcal{A}(I)$. We show that the restrictions of this map to the spaces of factorial, primal, and Glimm ideals yield homeomorphisms onto the corresponding ideal spaces of $\mathcal{A}(V)$. These results form the foundation for our subsequent analysis of ideal structures in $V \otimes^h B$.

Section $4$ examines the structure of space of $\epsilon$-ideals $\operatorname{Id}(V \otimes^h B)$ of $V \otimes^h B$. We define a weak topology $\tau_w$ on $\operatorname{Id}(V \otimes^h B)$ and study the fundamental map $\Phi$, which assigns pairs of ideals from $V$ and $B$ to $\epsilon$-ideals in $V \otimes^h B$. We first establish the continuity of $\Phi$ when restricted to 
$\operatorname{Prim}(V) \times \operatorname{Prim}(B)$, ensuring that it behaves well on primitive ideal spaces. Furthermore, we prove that for any proper ideals \( I_0 \) of \( V \) and \( J_0 \) of \( B \), the ideal \( V \otimes^{h} J_0 + I_0 \otimes^{h} B \) is an intersection of  primitive ideals of \( V \otimes^{h} B \), which, together with the continuity of the restricted map, paves the way for establishing the continuity of 
$\Phi$ on the entire space $\operatorname{Id}(V \otimes^h B)$.

In Section $5$, we expand our study of ideals in the Haagerup tensor product by introducing the concept of primal ideals of $V \otimes^h B$. We establish a structural characterization of closed minimal primal ideals, demonstrating that they are precisely of the form $I \otimes^h B+V \otimes^h J$  where $I$ and $J$ are minimal closed primal ideals in $V$ and $B$, respectively. We also show that this correspondence defines a homeomorphism between the spaces of such ideals. Additionally, we introduce the concept of factorial ideals in $V \otimes^h B$ and prove that the mapping $\Phi$, when restricted to the spaces of factorial ideals, is a homeomorphism onto $\operatorname{Fac}(V \otimes^h B)$. 

 Section $6$ develops the Glimm ideal theory for \(V \otimes^h B\), providing a complete characterization of its Glimm ideal space. We further establish that  the restriction of the mapping $\Phi$ to the spaces of Glimm ideals defines a homeomorphism onto $\operatorname{Glimm}(V \otimes^h B)$ when considered as sets of ideals, thereby extending the structural parallels as observed in the factorial and primal cases.  Lastly, we prove that $V\otimes^{h} B$ is quasi-standard if and only if both $V$ and $B$ are.

Throughout this paper, we will assume that $V$ denotes a TRO and $B$ a $C^{\ast}$-algebra. Additionally, unless otherwise stated, all ideals considered in the paper will be assumed to be closed. For simplicity of notation, an ideal in $V \otimes^h B$ will always mean an $\epsilon$-ideal.

\section{Preliminary results} 

Given two operator spaces $X$ and $Y$, the Haagerup norm of an element $x \in X \otimes Y$  is given  by

\[
\|x\|_h = \inf \left\{ \|a\| \|b\| : a = (a_{1j})_{1 \times n}, \, b = (b_{j1})_{n \times 1}, \, \text{and} \, x = \sum_{j=1}^n a_{1j} \otimes b_{j1} \right\}.
\]

 The Haagerup tensor product $X \otimes^h Y$ is the completion of \( X \otimes Y \) under the Haagerup norm (\cite{ER}). The following well known result regarding Haagerup tensor product will be used in this paper.  

 \begin{prop}
    (\cite{ER}, Proposition $9.2.5$) Let $X_i, Y_i, i= 1, 2$ be operator spaces and $f_i: X_i \to Y_i$ be two maps. Then 

     \begin{enumerate}
         \item If $f_1$ and $f_2$ are complete contractions, then so is $f_1 \otimes f_2 : X_1 \otimes^h X_2 \to Y_1 \otimes^h Y_2$. 
         \item If $f_1$ and $f_2$ are complete isometries, then so is $f_1 \otimes f_2 : X_1 \otimes^h X_2 \to Y_1 \otimes^h Y_2$. 
     \end{enumerate}
 \end{prop}

Let $V$ be a ternary ring of operators (TRO) between Hilbert spaces $\mathcal{H}$ and $\mathcal{K}$, i.e., a norm-closed subspace of $B(\mathcal{H}, \mathcal{K})$ that is closed under the ternary product $(x, y, z) \mapsto xy^*z, \quad \text{for all } x, y, z \in V$.  According to  \cite{MHA}, the  linking $C^{\ast}$-algebra $\mathcal{A}(V)$ generated by $V$, is given by $$A(V)=
\begin{bmatrix}
    C(V)       & V\\
    V^*   & D(V)
\end{bmatrix}$$
 
where $C(V)$ and $D(V)$ denotes the $C^{\ast}$-algebras generated by $VV^*$ and $V^*V$ respectively. 

A linear map $\pi: V \to B(\mathcal{H}, \mathcal{K})$ is called a \emph{representation} of $V$ if it preserves the ternary structure:
\[
\pi(xy^*z) = \pi(x)\pi(y)^{*}\pi(z), \quad \text{for all } x, y, z \in V.
\]
Given such a representation $\pi$, there exists an associated representation $\mathcal{A}(\pi): \mathcal{A}(V) \to B(\mathcal{H} \oplus \mathcal{K})$, which establishes a one-to-one correspondence between (irreducible) representations of $V$ and those of $\mathcal{A}(V)$ (see \cite{AAV}).

\begin{defn}
    Let $V$ be a TRO, and let $I$ be a closed subspace of $V$.

\begin{enumerate}
    \item The subspace $I$ is called an \emph{ideal} of $V$ if $IV^*V + VV^*I \subseteq I$.
The collection of all (proper) closed ideals of $V$ is denoted by $\operatorname{Id}(V)$ (resp., $\operatorname{Id}'(V)$). 
    The weak topology $\tau_w$ on $\operatorname{Id}(V)$ is generated by sub-basis sets of the form $U(J) = \{ I \in \operatorname{Id}(V) : I \not\supseteq J \},$
    where $J \in \operatorname{Id}(V)$.

    \item  An ideal \( P \) of \( V \) is prime if, for any ideals \( I, J, K \) in \( V \), the condition  
$IJK \subseteq P $  
implies that at least one of  
$I \subseteq P, \quad J \subseteq P, \quad \text{or} \quad K \subseteq P $
holds.
 The space of prime ideals of $V$ is denoted by $\operatorname{Prime}(V)$, equipped with the subspace topology inherited from $\operatorname{Id}(V)$.

    \item An ideal $I \in \operatorname{Id}(V)$ is called \emph{primitive} if it is the kernel of some irreducible representation of $V$. The space of primitive ideals of $V$ is denoted by $\operatorname{Prim}(V)$,  equipped with the subspace topology inherited from $\operatorname{Id}(V)$.
\end{enumerate}

\end{defn}

The natural map $\theta: \operatorname{Id}(V) \to \operatorname{Id}(\mathcal{A}(V)) $ is given by $$\theta(I)= \mathcal{A}(I).$$ The following proposition lists some known properties of the map $\theta$.

\begin{thm} \label{thm2} For a TRO $V$, the map $\theta$ has the following properties. 
    \begin{enumerate}
        \item[(i)] \textit{(\cite{AAV}, Theorem $2.6$)} The restriction of $\theta$ to $\operatorname{Prime}(V)$ is a homeomorphism onto $\operatorname{Prime}(\mathcal{A}(V))$.
        \item[(ii)] \textit{(\cite{AA}, Proposition $3$)} The restriction of $\theta$ to $\operatorname{Prim}(V)$ is a homeomorphism onto $\operatorname{Prim}(\mathcal{A}(V))$. 
    \end{enumerate}
        \item 
    
\end{thm}

For each subset \( S \) of \( V \), define  
\[
\operatorname{hull}(S) = \{ P \in \operatorname{Prim}(V) \mid S \subseteq P \}
\]  
and for each subset \( B \) of \( \operatorname{Prim}(V) \), let  
\[
\operatorname{ker}(B) = \bigcap \{ P \mid P \in B \}.
\]  

For an ideal \( I \) of a \( C^{\ast} \)-algebra, it is well known that \( \operatorname{ker}(\operatorname{hull}(I)) = I \). Using this fact along with the homeomorphism \( \theta \), the proof of the following result is immediate.  

\begin{prop}\label{result1}
    For any ideal \( I \) of \( V \), we have \( \operatorname{ker}(\operatorname{hull}(I)) = I \).
\end{prop}  

As a consequence of Proposition~\ref{result1}, we obtain 
\[
U(J) = \{ I \in \operatorname{Id}(V) \mid \operatorname{hull}(I) \cap V(J) \neq \emptyset \},
\]  
where \( V(J) = \{ P \in \operatorname{Prim}(V) \mid P \not\supseteq J \} \).

Given an ideal $I$ of $V$ and an ideal $J$ of $B$, the  natural ternary homomorphisms $q_I: V \to V/I$ and $q_J:B \to B/J$ induce two  complete contrations $q_I \otimes^h q_J : V \otimes^h B \to V/I \otimes^h B/J$ and $q_I \otimes^{\text{tmin}} q_J : V \otimes^{\text{tmin}} B \to V/I \otimes^{\text{tmin}} B/J$. The following well-known result will be used repeatedly.

\begin{prop} (\cite{AS}, Corollary $2.6$) For ideals $I$ and $J$ of $V$ and $B$, respectively, 
     $$\operatorname{ker}(q_I \otimes^h q_J)= V \otimes^h J+ I \otimes^h B.$$
\end{prop}

 We will consider two more tensor products  $V \otimes^{\text{tmin}} B$ and $V \otimes^{\text{tmax}} B$  where $V \otimes^{\text{tmin}} B$ denotes the completion of $V \otimes B$ with respect to the operator space injective norm. On the other hand,  $V \otimes^{\text{tmax}} B$  denotes the completion of $V \otimes B$ with respect to the norm given by  
 \[
\|x \|_{\text{tmax}} = \sup \{\|\theta_V\cdot \theta_B(x)\| : \theta_V, \theta_B \text{ are representations of $V$ and $B$, respectively } \}.\]

where more details about $\otimes^{\text{tmax}}$ can be found in \cite{MK_JR}. 

 \begin{prop}
     The canonical map $q:V \otimes^{\text{tmax}} B \to V \otimes^{\text{tmin}} B $ is a completely contractive quotient map.
 \end{prop}
The following result ensures that $V \otimes^h B$ can be considered as a subspace of $V\otimes^{\text{tmin}} B$ with a different norm.

 \begin{prop} (\cite{AKKKK}, Proposition $4.9$)
Let $V$ be a TRO and $B$ a $C^{\ast}$-algebra. Then the identity map $\epsilon: V \otimes^h B \to V \otimes^{\text{tmin}} B$ is injective.  
\end{prop} 

In \cite{AKKKK}, the notions of \(\epsilon\)-ideals and \(i\)-ideals were introduced. For the reader's convenience, we recall their definitions.

\begin{defn}  Let $P$ be a closed subspace of $V \otimes^h B$.
\begin{enumerate}
    \item[(i)] $P$ is called an \(\epsilon\)-ideal if there exists a closed ideal \( Q \) of \( V \otimes^{\text{tmin}} B \) such that \( P = \epsilon^{-1}(Q) \), where \( \epsilon: V \otimes^h B \to V \otimes^{\text{tmin}} B \) is the canonical injective map. Equivalently,  \( P \) is an \(\epsilon\)-ideal precisely when \( P = Q \cap (V \otimes^h B) \) for some closed ideal \( Q \) in \( V \otimes^{\text{tmin}} B \).  

\item[(ii)] $P$ is called an \( i \)-ideal if there exists a closed ideal \( Q \) of \( \mathcal{A}(V) \otimes^h B \) such that \( P = i^{-1}(Q) \), where \( i: V \otimes^h B \to \mathcal{A}(V)\otimes^h B \) is the natural isometry. Equivalently, \( P \) is an \( i \)-ideal if and only if \( P = Q \cap (V \otimes^h B) \) for some closed ideal \( Q \) in \( \mathcal{A}(V) \otimes^h B \).  
\end{enumerate}

\end{defn}

It has been established that a closed subspace \( P \) of \( M \otimes^h B \) is an \(\epsilon\)-ideal if and only if it is an \( i \)-ideal (\cite{AKKKK}, Proposition $4.12$). The following notion of primitive ideals in $V \otimes^h B$ was introduced in \cite{AA}.

\begin{defn}
    A linear map $\pi: V \otimes^h B \to B(\mathcal{H}, \mathcal{K})$ is called a (irreducible) representation of $V \otimes^h B$ if there exists a (irreducible) representation $\rho$ of $V \otimes^{\text{tmax}} B$ such that $\pi= \rho  \epsilon' $ where $\epsilon': V \otimes^h B \to V \otimes^{\text{tmax}} B$ is natural injective map. A closed subspace $P$ of $V \otimes^h B$ is called a primitive ideal if $P$ is the kernel of some irreducible representation  of $V \otimes^h B$.
\end{defn}

\begin{defn}
    Given a topological space $X$, we will denote by $C^b(X)$ the algebra of bounded continuous complex valued functions on $X$. The complete regularization $\rho X$ of $X$ is the quotient space obtained by identifying points $x_1, x_2 \in X$ whenever  $f(x_1) = f(x_2) \quad \text{for all }  f \in C^b(X)$. The quotient map $\rho_X: X \to \rho X$ sends each $x \in X$ to its equivalence class $[x]$ in $\rho X$. For each function $f \in C^b(X)$, we define its induced function $f^\rho$ on $\rho X$ by  $f^\rho([x]) = f(x)$. The topology of $\rho X$, denoted by $\tau_{cr}$, is induced by the function $\{f^{\rho}: f \in C^b(X) \}$. It is known that the space $\rho X$ is completely regular and the collection of cozero sets  
\[
\mathrm{coz}(f^\rho) = \{ [x] \in \rho X : f^\rho([x]) \neq 0 \}, \quad f \in C^b(X),
\]  
forms a basis for the $\tau_{cr}$ topology (\cite{rings}, Theorem $3.9$). 
\end{defn}

Let $X$ and $Y$ be topological spaces. We denote by $\rho X \times \rho Y$, the product space equipped with the product topology $\tau_p$, and by $\rho (X \times Y)$ the space endowed with the topology $\tau_{cr}$. The following result, that establishes a natural identification between $\rho X \times \rho Y$ and $\rho (X \times Y)$ as sets, will be useful in the final section of this paper. 

\begin{prop}
    (\cite{Glimm}, Lemma $1.1$) For topological spaces $X$ and $Y$, there is an open bijection $\rho X \times \rho Y \to \rho (X \times Y)$ sending $(\rho_X(x), \rho_Y(y)) \to \rho_{X \times Y}(x, y)$.
\end{prop}

Another natural topology on $\rho X$ is the quotient topology $\tau_q$, induced by the map $\rho_X$. Since $\rho_X$ is continuous as a map into $(\rho X, \tau_{cr})$, $\tau_q$ is always finer than $\tau_{cr}$.

The space of primitive ideals $\mathrm{Prim}(V)$ and $\mathrm{Prim}(V \otimes^h B)$ is not necessarily completely regular in the $\tau_w$-topology. To address this, we employ complete regularization, defining the Glimm ideal space as the quotient modulo inseparability by continuous functions. By a Glimm ideal, we mean the ideal given by taking the intersection of the primitive ideals in a given equivalence class. This construction provides a natural framework for analyzing the behavior of Glimm ideals under the Haagerup tensor product.

\section{Ideals of a TRO}  %% Please avoid complex formulas in (sub)titles

In the study of ideals, the concept of primal, factorial, and Glimm ideals have been well-explored in various contexts, particularly for $C^{\ast}$-algebras. However, these concepts in the setting of TROs have not been as thoroughly investigated. In this section, we explore these three classes of ideals within TROs. Although most of the results are straightforward and follow directly from known results on the linking $C^{\ast}$-algebra of a TRO, we include the proofs here for completeness. The results presented will be crucial for understanding the primal, factorial and Glimm ideals of $V \otimes^h B$ in the later sections of the paper.

\begin{defn}

A representation \( \pi: V \to B(\mathcal{H}, \mathcal{K}) \) is called a \textit{factorial representation} if the von Neumann algebra \( \mathcal{M} \) generated by \( \{ \pi(v) : v \in V \} \subseteq B(\mathcal{H} \oplus \mathcal{K}) \) is a factor, i.e., its center consists only of scalar multiples of the identity:
\[
Z(\mathcal{M}) = \mathbb{C} (I_{\mathcal{H}} \oplus I_{\mathcal{K}}).
\]
    
\end{defn}

Notably, the von Neumann algebra generated by \( \{ \pi(v) : v \in V \} \subseteq B(\mathcal{H} \oplus \mathcal{K}) \) coincides with the von Neumann algebra generated by \( \{ \mathcal{A}(\pi)(x) : x \in \mathcal{A}(V) \} \subseteq B(\mathcal{H} \oplus \mathcal{K}) \). This leads to the following:

\begin{prop}
A representation \( \pi \) of \( V \) is a factorial representation if and only if the associated representation \( \mathcal{A}(\pi) \) of \( \mathcal{A}(V) \) is a factorial representation.
\end{prop}

\begin{defn}
    An ideal $I$ of $V$ is called a \emph{factorial} ideal of $V$ if $I= \operatorname{ker}(\pi),$ for some  factorial representation $\pi$ of $V$.
\end{defn}

The following result follows immediately from Proposition \ref{result1} and (\cite{AAV}, Lemma $2.7$). 

\begin{prop} \label{facmain}
    An ideal $I$ of $V$ is factorial if and only if $\mathcal{A}(I)$ is factorial ideal of $\mathcal{A}(V)$.
\end{prop}

\begin{defn}
  
 An ideal $P$ of a TRO $V$ is called primal if, whenever $J_1, J_2,  \newline \dots, J_{2n+1}$, $n\in \ N$, are ideals in $V$ with $J_1J_2 \dotsm  J_{2n+1} = \{0\}$, then $J_k \subseteq P$ for at least one value of $k$. We  denote the set of all primal ideals of $V$ by $\operatorname{Primal}(V)$.
  
\end{defn}

\begin{rem}
    It is easy to see that every  prime ideal of a TRO is primal. Let $P$ and $Q$ be ideals of a TRO $V$ such that $P \subseteq Q$. If  $P$ is primal then $Q$ is also Primal.
\end{rem}

 The following proposition establishes that the restriction of $\theta$ to primal ideals is also a bijection.

\begin{prop} \label{primalmain}
    $P$ is a primal ideal of $V$ if and only if $\mathcal{A}(P)$ is a primal ideal of $\mathcal{A}(V)$.
\end{prop}

\begin{proof}
 Let $P$ be a primal ideal of $V$. Suppose $J_1, J_2, \dotsc, J_n$ are ideals of $\mathcal{A}(V)$ satisfying $J_1 J_2 \dotsm J_n = \{0\}$. Since $J_i = \mathcal{A}(I_i)$ for some ideal $I_i$ of $V$, we have
    \begin{align*}
        \mathcal{A}(I_1)\mathcal{A}(I_2) \dotsm \mathcal{A}(I_n) &= \mathcal{A}(I_1) \cap \mathcal{A}(I_2) \cap \dotsm \cap \mathcal{A}(I_n) \\
        &= \mathcal{A}(I_1 \cap I_2 \cap \dotsm \cap I_n) = \{0\}.
    \end{align*}
    Hence, $I_1 \cap I_2 \cap \dotsm \cap I_n = \{0\}$. Now, if $n$ is odd, then by (\cite{AA}, Lemma $1$), $I_1I_2 \dotsm I_n = I_1 \cap I_2 \cap \dotsm \cap I_n = \{0\}$, therefore $I_k \subseteq P$ for some $k$, implying $\mathcal{A}(I_k) \subset \mathcal{A}(P)$. In the case where $n$ is even, $I_1 \cap I_2 \cap \dotsm \cap I_n = \{0\}$ implies $I_1 \cap I_1 \cap I_2 \cap \dotsm \cap I_n = \{0\}$, so by a reasoning similar to that in the case where $n$ is odd, we are done.

    Conversely, assume that $\mathcal{A}(P)$ is a primal ideal of $\mathcal{A}(V)$, and $I_1, I_2, \dotsc, I_{2 n+1}$ are ideals of $V$ satisfying $I_1 \cap I_2 \cap \dotsm \cap I_{2n+1} = \{0\}$. Then $\mathcal{A}(I_1)\mathcal{A}(I_2) \dotsm \mathcal{A}(I_{2 n+1}) = \{0\}$, implying $\mathcal{A}(I_k) \subseteq \mathcal{A}(P)$ for some $k$, which gives $I_k \subseteq P$.
\end{proof}

\pagebreak
The following proposition provides a characterization of primal ideals in terms of the convergence of primitive ideals.

\begin{prop}\label{prop34}
 $P$ is a Primal ideal of $V$ if and only if there exists a net $P_{\alpha} \in \operatorname{Prim}(V)$ such that $P_{\alpha}$ converges to $Q$ in $\tau_w$-topology for all $Q \in \operatorname{hull}(P)$.
\end{prop}

\begin{proof}
Suppose $P$ is a primal ideal in $V$. By the last proposition, $\mathcal{A}(P)$ is a primal ideal in $\mathcal{A}(V)$. Hence, there exists a net $\mathcal{A}(P_{\alpha}) \subseteq \operatorname{Prim}(\mathcal{A}(V))$ such that $\mathcal{A}(P_{\alpha})$ converges to $\mathcal{A}(Q)$ for all $\mathcal{A}(Q) \in \operatorname{hull}(\mathcal{A}(P))$, so $P_{\alpha}$ converges to $Q$ for all $Q \in \operatorname{hull}(P)$. The proof of the converse is along similar lines using the last proposition, and we leave the details to the reader.

\end{proof}

For a $C^{\ast}$-algebra $A$, the space $\operatorname{Glimm}(A)$ of the Glimm ideals of $A$ is naturally obtained through the complete regularization of $\operatorname{Prim}(A)$ (\cite{ARR}). Following the similar approach, we now define the space of Glimm ideals for a TRO $V$.

\begin{defn}
  
Let $V$ be a TRO. We define $\operatorname{Glimm}(V)$ as the complete regularization $\rho \operatorname{Prim}(V)$ of $\ \operatorname{Prim}(V)$, and denote by $\rho_V: \operatorname{Prim}(V) \to \operatorname{Glimm}(V)$ the complete regularization map. There is  a one-to-one correspondence between the quotient space \(\operatorname{Prim}(V) / \approx\) and a set of ideals in \(V\), given by 
\[
[P] \mapsto \operatorname{k}([P]) = \bigcap \{ Q \in \operatorname{Prim}(V) : Q \approx P \}.
\]

The ideals obtained through this correspondence are called Glimm ideals of $V$.  We will regard elements of $\operatorname{Glimm}(V)$ as either points of a topological space or as ideals of $V$, depending on the context.
  
\end{defn}

The proof of the following lemma is straightforward.
\begin{lem}
 $P \approx Q$ if and only if $\mathcal{A}(P) \approx \mathcal{A}(Q)$ for all primitive ideals $P$ and $Q$ of a TRO $V$. Therefore, $\mathcal{A}([P])= [\mathcal{A}(P)]$. 
\end{lem}

 Using the above lemma, we obtain the following.

\begin{prop}\label{mainglimm1} For a TRO $V$, $I$ is a Glimm ideal of $V$ if and only if $\mathcal{A}(I)$ is a Glimm ideal of $\mathcal{A}(V)$.
\end{prop}

\begin{proof}
Let \( I \) be a Glimm ideal of \( V \). By definition, there exists a primitive ideal \( P \) of \( V \) such that  $I = k([P])$.
Applying the functor \( \mathcal{A} \), $\mathcal{A}(I) = \mathcal{A}(k([P]))$. By  previous lemma, we know that  
$\mathcal{A}(I) = k([\mathcal{A}(P)])$. This shows that \( \mathcal{A}(I) \) is a Glimm ideal of \( \mathcal{A}(V) \).  

Conversely, suppose \( \mathcal{A}(I) \) is a Glimm ideal of \( \mathcal{A}(V) \). Then, by definition, there exists a primitive ideal \( P \) of \( V \) such that  $\mathcal{A}(I) = k([\mathcal{A}(P)])=\mathcal{A}(k([P]))$. Since \( \theta \) is injective, it follows that $I = k[P]$. Hence, \( I \) is a Glimm ideal of \( V \), completing the proof. 

\end{proof}

\begin{defn}
    A net $(e_{\lambda}, f_{\lambda})$ in a TRO $V$ is called an \textit{approximate identity} for $V$ if there exists a constant $K > 0$ such that  $\| e_{\lambda} \| \leq K, \quad \| f_{\lambda} \| \leq K, \quad \text{for all } \lambda$. Moreover, for every $x \in V$, the nets  $ e_{\lambda} f_{\lambda}^{\ast} x $ and $ x e_{\lambda}^{\ast} f_{\lambda}$ converge to $x$.  

A TRO $V$ is called \textit{$\sigma$-unital} if it admits a countable approximate identity.  
\end{defn}

\begin{cor} \label{sector3main}
Let $V$ be a TRO satisfying one of the following conditions:  
\begin{enumerate}  
    \item[(i)] $\operatorname{Prim}(V)$ is compact,  
    \item[(ii)] The complete regularization map $\rho_V : \operatorname{Prim}(V) \to \operatorname{Glimm}(V)$ is open,   
    \item[(iii)] $V$ is $\sigma$-unital and $\operatorname{Glimm}(V)$ is locally compact.  
\end{enumerate}  
Then the complete regularization $\rho(\operatorname{Prim}(V) \times \operatorname{Prim}(B))$ of $\operatorname{Prim}(V) \times \operatorname{Prim}(B)$ is homeomorphic to the product space $\operatorname{Glimm}(V) \times \operatorname{Glimm}(B)$ for any $C^{\ast}$-algebra $B$.   
\end{cor}

\begin{proof}
This follows from the corresponding result for $C^*$-algebras ( \cite{Glimm}, Proposition 1.9) and the fact that $V$ and its linking  $C^*$-algebra $\mathcal{A}(V)$ satisfy:
\begin{enumerate}
    \item[(i)] $ \operatorname{Prim}(V) \cong \operatorname{Prim}(\mathcal{A}(V)) $,
    \item[(ii)] $ \operatorname{Glimm}(V) \cong \operatorname{Glimm}(\mathcal{A}(V)) $,
    \item[(iii)] $V$ is $\sigma$-unital if and only if $\mathcal{A}(V)$ is $\sigma$-unital.
\end{enumerate}
Applying the known result to $ \mathcal{A}(V) $ and $ B $ completes the proof.
\end{proof}

Let \(V\) be a TRO. In addition to the equivalence relation \(\approx\) introduced in Definition, we now define a second equivalence relation on the primitive ideal space \(\mathrm{Prim}(V)\), based on the topological separation properties of primitive ideals.

\begin{defn}
For \(P, Q \in \mathrm{Prim}(V)\), we write $
P \sim Q$
if and only if there do not exist disjoint open subsets \(U_1, U_2 \subseteq \mathrm{Prim}(V)\) such that \(P \in U_1\) and \(Q \in U_2\). That is, \(P\) and \(Q\) cannot be separated by disjoint open sets in the hull-kernel topology.
\end{defn}

Note that the relation \(\sim\) is an equivalence relation on \(\mathrm{Prim}(V)\). If \(P \sim Q\)  then  \(P \approx Q\). However, the converse does not always hold, see \cite{qs1991}.

\begin{defn}

A TRO \( V \) is said to be quasi-standard if the equivalence relation
$\sim$
on its primitive ideal space \( \mathrm{Prim}(V) \) is open.  
\end{defn}

%The following result follows directly from Theorem 2.2.

\begin{prop}
    Let \( V \) be a TRO. Then \( V \) is quasi-standard if and only if its linking algebra \( \mathcal{A}(V) \) is quasi-standard.
\end{prop}

\begin{proof}
Note that $\operatorname{Prim}(V) \cong \operatorname{Prim}(\mathcal{A}(V))$ by Theorem \ref{thm2}. It follows easily that $P \sim Q$ if and only if $\mathcal{A}(P)\sim \mathcal{A}(Q) $ for all $P,Q\in \operatorname{Prim}(V) $.  Furthermore, the relation \(\sim\) is open on \(\operatorname{Prim}(V)\) if and only if it is open on \(\operatorname{Prim}(\mathcal{A}(V))\). Hence the result follows.
\end{proof}
%Let \( V \) be a TRO. Then \( V \) is quasi-standard if and only if its linking algebra \( \mathcal{A}(V) \) is quasi-standard.

The following proposition follows from the fact that $\operatorname{Prim}(V) \cong \operatorname{Prim}(\mathcal{A}(V))$, $\operatorname{Primal}(V) \cong \operatorname{Primal}(\mathcal{A}(V))$, $\operatorname{Glimm}(V) \cong \operatorname{Glimm}(\mathcal{A}(V))$, and together with the  existing results concerning $C^*$-algebras applied to $\mathcal{A}(V)$ \cite{qs1991}.

\begin{prop}

Let \( V\) be a TRO. Then the  following  statements are equivalent:

\begin{enumerate}
    \item[(i)] \( V \) is quasi-standard;
    \item[(ii)] The complete regularization map \( \rho_V\) is open and every Glimm ideal of \( V \) is primal.
    \item[(iii)] Every Glimm ideal of \( V\) is primal and the topologies \( \tau_q \) and \( \tau_{cr} \) on \( \mathrm{Glimm}(V)\) coincide with the topology \( \tau \) on \( \mathrm{Min\text{-}Primal}(V) \).
    \item[(iv)] Every Glimm ideal of \( V \) is primal and the topology \( \tau_q \) on \( \mathrm{Glimm}(V) \) coincides with the topology \( \tau \) on \( \mathrm{Min\text{-}Primal}(V) \).
\end{enumerate}

\end{prop}

\begin{ex}
 
For an infinite dimensional separable Hilbert spaces $\mathcal{H}$ and $\mathcal{K}$,  it follows from the equivalent condition above that $B(\mathcal{H},\mathcal{K})$ is quasi-standard.
   
\end{ex}

\section{The space of ideals of $V \otimes^h B$}

For a TRO $V$ and a $C^*$-algebra $B$,  let $ \operatorname{Id}(V \otimes^h B)$ denote the collection of $\varepsilon$-ideals of  $V \otimes^h B$. We define the weak topology $\tau_w$ on $ \operatorname{Id}(V \otimes^h B)$ by specifying as  sub-basic sets all sets of the form
\[
U(K)=\{ I \in  \operatorname{Id}(V \otimes^h B) : K \nsubseteq I \}
\]
where $K \in \operatorname{Id}(V \otimes^h B)$.  Note that for $I \in \operatorname{Id}(V)$ and $J \in \operatorname{Id}(B)$, $V \otimes^h J + I \otimes^h B \in \operatorname{Id}(V \otimes^h B)$ (\cite{AAV}, Theorem $4(b))$. Thus,  the map $$\Phi: \operatorname{Id}(V) \times \operatorname{Id}(B) \to \operatorname{Id}(V \otimes^h B)
$$  given by
\[
\Phi(I, J) = V \otimes^h J + I \otimes^h B
\] 

is  well defined. In this section, we prove that $\Phi$ has many nice  topological properties. To that end, consider a map 
\[
\Phi_1: \operatorname{Id}(V) \times \operatorname{Id}(B) \to \operatorname{Id}(V\otimes^{\text{tmin}} B),
\]
defined as
\[
\Phi_1(I, J) = \ker(q_I \otimes^{\text{tmin}} q_J),
\]
where  $q_I \otimes^{\text{tmin}} q_J$ is the canonical homomorphism from $V \otimes^{\text{tmin}} B$ onto $(V/I) \otimes^{\text{tmin}} (B/J)$. Furthermore, there is a map
\[
\Phi_2: \operatorname{Id}(V \otimes^{\text{tmin}} B) \to \operatorname{Id}(V\otimes^h B)
\]
given  by \[\Phi_2(K)= \epsilon^{-1}(K).\]
Since $\epsilon$ is an inclusion, $\Phi_2(K)$ is nothing but $K \cap (V \otimes^h B)$. The map $\Phi_2$ is $\tau_w$-continuous because for $K \in \operatorname{Id}(V\otimes^h B) $ and $I \in \operatorname{Id}(V \otimes^{\text{tmin}} B)$, $\Phi_2(I) \supseteq K$ if and only if $I$ contains the closed ideal of $\operatorname{Id}(V \otimes^{\text{tmin}} B)$ generated by $\epsilon(K)$. Moreover, $\Phi_2$ maps 
$\text{Prim}(V \otimes^{\text{tmin}} B)$ into $\text{Prim}(V \otimes^\text{h} B)$ 
because every irreducible representation  $\pi$ of $V \otimes^{\text{tmin}} B$ corresponds to irreducible representation  $\pi \circ q$ of 
$V \otimes^{\text{tmax}} B$, as $q$ is surjective. This further corresponds to irreducible 
$*$-representation $\pi \circ q \circ \epsilon$ of $V \otimes^{\text{h}} B$.

 \begin{lem}\label{result3} Let $V$ be a TRO and $B$ be a $C^{\ast}$-algebra.  Let $I$ and $J$ be ideals of $V$ and $B$ respectively. Then $$\ker(q_I \otimes^ h q_J) = \ker(q_I \otimes^{\text{tmin}} q_J) \cap (V \otimes^h B).$$
    The commutativity of the following diagram follows from this equality:
\[
\xymatrix{
 \operatorname{Id}(V) \times \operatorname{Id}(B) \ar[r]^{\Phi_1} \ar[dr]_{\Phi} & \operatorname{Id}(V \otimes^{\text{tmin}} B)\ar[d]^{\Phi_2} \\
 & \operatorname{Id}(V\otimes^h B)
}
\]

\end{lem}

\begin{proof}

Let $x \in V \otimes^h B$, and let $(x_\lambda)$ be a net in $V \otimes B$ such that $\lim_\lambda \| x_\lambda - x \|_h = 0$. Since $\|.\|_{\text{tmin}} \leq \|.\|_h$, 
$\lim_\lambda \| x_\lambda - x \|_{\text{tmin}} = 0$. As the maps $q_i \otimes^h q_J$ and $q_I \otimes^{\text{tmin}} q_J$ are bounded, it follows that $(q_I \otimes^h q_J)(x_\lambda)$ and $(q_I \otimes^{\text{tmin}} q_J)(x_\lambda)$ converges to $q_I \otimes^h q_J(x)$ and $(q_I \otimes^{\text{tmin}} q_J)(x)$ respectively. However, since $q_I \otimes^h q_J$ and $q_I \otimes^{\text{min}} q_J$ both agree on $V \otimes B$, we have $q_I \otimes^h q_J(x)= q_I \otimes^{\text{tmin}} q_J(x)$. Hence, we obtain $(q_I \otimes^{\text{tmin}} q_J) \vert_{V \otimes^h B}= q_I \otimes^h q_J$. From this, the result follows.

\end{proof}

The map $\theta \times 1:\operatorname{Prim}(V) \times \operatorname{Prim}(B) \to \operatorname{Prim}(\mathcal{A}(V)) \times \operatorname{Prim}(B) $ is a homeomorphism. Moreover, the map $\Psi: \operatorname{Prim}(V \otimes^{\text{tmin}} B) \to \operatorname{Prim}(\mathcal{A}(V) \otimes^{\text{min}} B)$ is also a homeomorphism (\cite{MK_JR}, Proposition $3.1$). Furthermore, the map $\Phi': \operatorname{Prim}(\mathcal{A}(V)) \times \operatorname{Prim}(B) \to \operatorname{Prim}(\mathcal{A}(V) \otimes^{\text{min}} B)  $ is a homeomorphism onto its image (\cite{Gui}, Theorem $5$). Thus, we obtain the following commutative diagram

\[
\begin{tikzcd}
  \operatorname{Prim}(V) \times \operatorname{Prim}(B)
    \arrow[r, "\Phi_1"]
    \arrow[d, swap, "\theta \times 1"]
  &
  \operatorname{Prim}(V \otimes^{\text{tmin}} B)
    \arrow[d, swap, "\Psi"]
  \\
  \operatorname{Prim}(\mathcal{A}(V)) \times \operatorname{Prim}(B)
    \arrow[r, swap, "\Phi'"]
  &
  \operatorname{Prim}(\mathcal{A}(V)\otimes^{\mathrm{min}} B)
\end{tikzcd}
\]
 Using the commutativity of the diagram, we conclude that restriction of $\Phi_1$ to \newline $\operatorname{Prim}(V) \times \operatorname{Prim}(B)$  is a homeomorphism onto its image. Additionally,   the map \(\Phi\) sends \(\operatorname{Prim}(V) \times \operatorname{Prim}(B)\) to \(\operatorname{Prim}(V \otimes^h B)\) (\cite{AA}, Theorem $4$). Finally, since \(\Phi\) factors as \(\Phi = \Phi_2 \circ \Phi_1\), we obtain the following result.

\begin{prop} \label{fourmain}
    The restriction of $\Phi$ to $\text{Prim}(V) \times \text{Prim}(B)$ is continuous for the 
product \textit{hk}-topology on $\text{Prim}(V) \times \text{Prim}(B)$ and the $\tau_w$-topology on $\text{Prim}(V \otimes^h B)$.
\end{prop}

\begin{lem}\label{result4}

Let $I_0$ be a proper ideal of $V$ and $J_0$ a proper ideal of $B$. Then, 

\[
V\otimes^{h} J_0 + I_0 \otimes^{h} B = \bigcap \{ V\otimes^{h} Q + P \otimes^{h} B : P \in hull(I_0), Q\in hull(J_0) \}.
\]
\end{lem}

\begin{proof}

By the preceding Lemma \ref{result3}, it suffices to show that
for $I_0\in \operatorname{Id}(V) $ and $J_0 \in \operatorname{Id}(B)$, we have
\[\ker(q_{I_0} \otimes^{\text{tmin}} q_{J_0}) 
= \bigcap\{ \ker(q_P \otimes^{\text{tmin}} q_Q)  : (P, Q) \in \text{hull}(I_0) \times \text{hull}(J_0) \} .
\] 
Since the map \( q_{I_0} \otimes^{\text{tmin}} q_{J_0} : V\otimes^{\text{tmin}} B \to V/I_0 \otimes^{\text{tmin}} B/J_0 \) is a ternary homomorphism, applying the functor $\mathcal{A}$ and using (\cite{AKKKK}, Prop $4.6$), we obtain the $C^{\ast}$-homomorphism 

\[
 q_{\mathcal{A}(I_0)} \otimes^{\text{min}} q_{J_0}  : \mathcal{A}(V) \otimes^{\text{min}} B \to \mathcal{A}(V)/\mathcal{A} (I_0) \otimes^{\text{min}} B/J_0
\]

Using  (\cite{Lazer}, Remark $2.4$)  and (\cite{AAV}, Theorem $2.6(4)$), we obtain
 $$\ker( q_{\mathcal{A}(I_0)} \otimes^{\text{min}} q_{J_0})=  \bigcap\{ \ker(q_{\mathcal{A}(P)} \otimes^{\text{min}} q_Q)  : (P, Q) \in \text{hull}(I_0) \times \text{hull}(J_0) \}.$$

But since  $\ker( q_{\mathcal{A}(I_0)} \otimes^{\text{min}} q_{J_0}) = \mathcal{A} (\ker(q_{I_0} \otimes^{\text{tmin}} q_{J_0}))$ (\cite{AAV}, Lemma $2.7$) and $\mathcal{A}$ is injective on ideals, it follows that \[
\ker(q_{I_0} \otimes^{\text{tmin}} q_{J_0}) = \bigcap \{ \ker(q_P \otimes^{\text{tmin}} q_Q) : (P, Q) \in \operatorname{hull}(I_0) \times \operatorname{hull}(J_0) \}.
\]  
  
\end{proof}

In (\cite{AA}, Theorem $4(a)$),  it was shown that if $I$ and $J$ are primitive ideals of $V$ and $B$, respectively, then the ideal \( V \otimes^{h} J + I \otimes^{h} B \) is a primitive ideal of \( V \otimes^{h} B \). The last result shows that, for any proper ideals \( I_0 \) of \( V \) and \( J_0 \) of \( B \), the ideal \( V \otimes^{h} J_0 + I_0 \otimes^{h} B \) is an intersection of  primitive ideals of \( V \otimes^{h} B \) belonging to $\Phi(\operatorname{Prim}(V) \times \operatorname{Prim}(B))$.
Utilizing this fact along with Proposition \ref{fourmain}, the proof of the subsequent theorem follows the same reasoning as in (\cite{Kanimin}, Lemma 1.5). For completeness, we provide a detailed proof.

\begin{thm} \label{mainthm1} The mapping $\Phi: \operatorname{Id}(V) \times \operatorname{Id}(B) \to \operatorname{Id}(V \otimes^h B)$ is $\tau_w$-continuous.
\end{thm}

\begin{proof}
      We need  to show that if $K \in \operatorname{Id}(V \otimes^h B)$ and $I_0 \in \operatorname{Id}(V)$, $J_0 \in \operatorname{Id}(B)$ satisfy $\Phi(I_0, J_0) \in U(K)$,  there exist $\tau_w$-neighbourhoods $V$ of $I_0$ and $W$ of $J_0$ in $\operatorname{Id}(V)$ and $\operatorname{Id}(B)$, respectively, such that $\Phi(V \times W) \subseteq U(K)$. 

 Since $I_0$ and $J_0$ are necessarily proper, it follows from  Lemma \ref{result4} that $V\otimes^h J_0 + I_0 \otimes^h B$ is the intersection of ideals of the form
\[
V \otimes^h Q + P \otimes^h B,
\]
where $P \in h(I_0)$ and $Q \in h(J_0)$. Thus, since $K \nsubseteq V \otimes^h J_0+ I_0 \otimes^h B$,  there exists $P_0 \in h(I_0)$ and $Q_0 \in h(J_0)$ such that
 $K \nsubseteq (V \otimes^h Q_0 + P_0 \otimes^h B)$.
Since $\Phi$ is continuous on $\text{Prim}(V) \times \text{Prim}(B)$ (Proposition \ref{fourmain}) and
\[
\Phi(P_0, Q_0) \in U(K) \cap \text{Prim}(V \otimes^h B),
\]
there exist open neighbourhoods $V_0$ of $P_0$ in $\text{Prim}(V)$ and $W_0$ of $Q_0$ in $\text{Prim}(B)$ such that $\Phi(V_0 \times W_0) \subseteq U(K)$. Consequently, $K \nsubseteq V \otimes^h Q + P \otimes^h B$ for all $P \in V_0$ and $Q \in W_0$. Now, define
\[
U_1 = \{I \in \operatorname{Id}(V): h(I) \cap V_0 \neq \emptyset\}, \quad U_2 = \{J \in \operatorname{Id}(B): h(J) \cap W_0 \neq \emptyset\}.
\]
Then $I_0 \in U_1$ and $J_0 \in U_2$, and $U_1$,  $U_2$ are $\tau_w$-open in $\operatorname{Id}(V)$ and  $\operatorname{Id}(B)$, respectively. Finally, let $I \in U_1$ and $J \in U_2$. Then, there exist $P \in h(I) \cap V_0$ and $Q \in h(J) \cap W_0$. Then
\[
V \otimes^h J + I \otimes^h B \subseteq V \otimes^h Q + P \otimes^h B \quad \text{and} \quad  K \nsubseteq V \otimes^h Q + P \otimes^h B.
\]
This establishes that $V \otimes^h J + I \otimes^h B \in U(K)$, completing the proof. 
\end{proof}

Consider   the well-defined maps  \[
\Psi_V :  V\otimes^h J + I \otimes^h B \to I \quad \text{and} \quad \Psi_B : V \otimes^h J + I \otimes^h B \to J
\]
 from $\Phi(\operatorname{Id}'(V) \times \operatorname{Id}'(B))$ onto $\operatorname{Id}'(V)$ and $\operatorname{Id}'(B)$, respectively. Observe that  
\[
\Psi_V^{-1}(U(K)) = U(K \otimes^h B) \cap \Phi(\operatorname{Id}'(V) \times \operatorname{Id}'(B)).
\]
This implies that \( \Psi_V \) is \( \tau_w \)-continuous. Similarly, \( \Psi_B \) is also \( \tau_w \)-continuous. As an immediate consequence of this we therefore obtain:

\begin{thm}\label{mainresult2}
The restriction of $\Phi$ to $\operatorname{Id}'(V) \times \operatorname{Id}'(B)$ is a homeomorphism onto its image in $\operatorname{Id}'(V \otimes^h B)$.
\end{thm}

The space $\operatorname{Prim}(V \otimes^h B)$ is not necessarily Hausdorff in general. The following corollary provides a necessary and sufficient condition for it to be Hausdorff in terms of $\operatorname{Prim}(V)$ and $\operatorname{Prim}(B)$.

\begin{cor}
    \textit{For a  TRO $V$ and a $C^*$-algebra $B$, the following conditions are equivalent:}
\begin{enumerate}
    \item[(i)] $\operatorname{Prim}(V)$ and $\operatorname{Prim}(B)$ are both Hausdorff.
    \item[(ii)] $\operatorname{Prim}(V \otimes^h B)$ is Hausdorff.
    \end{enumerate}
\end{cor}

\begin{proof}
    Suppose that $\operatorname{Prim}(V)$ and $\operatorname{Prim}(B)$ are both Hausdorff. Then every prime ideal of either $V$ or $B$ is primitive. Consequently, the canonical map  
\[
\Phi: \operatorname{Prim}(V) \times \operatorname{Prim}(B) \to \operatorname{Prim}(V \otimes^h B)
\]
is homeomorphic onto  $\operatorname{Prim}(V \otimes^h B)$ by (\cite{AA}, Theorem $4(b)$) and Theorem \ref{mainthm1}). Since $\operatorname{Prim}(V) \times \operatorname{Prim}(B)$ is Hausdorff, thus $\operatorname{Prim}(V \otimes^h B)$ is Hausdorff.

Conversely, if $\operatorname{Prim}(V \otimes^h B)$ is Hausdorff, then so is $\Phi(\operatorname{Prim}(V) \times \operatorname{Prim}(B))$ and hence so are $\operatorname{Prim}(V)$ and $\operatorname{Prim}(B)$.

\end{proof}

\section{Primal and Factorial ideals of $V \otimes^h B$}

In this section our first aim is to give a complete classification of primal ideals of $V \otimes^h B$. We begin with the definition.

\begin{defn}
     An $\epsilon$-ideal $P$ of $V \otimes^h B$ is called a  primal ideal if for any collection of $\epsilon$-ideals $I_1, I_2, \cdots ,I_{2n+1}$  of $V \otimes^h B$ satisfying $$I_1 I_2...I_{2n+1}=0,$$ then there exist atleast one index $i$ such that $I_i \subseteq P$ . We denote the set of all primal ideals of $V \otimes^h B$  by $\operatorname{Primal}(V \otimes^h B)$.

\end{defn}

We now examine the conditions under which sum of product ideals in $V \otimes^h B$ are  primal.

\begin{thm}\label{mainthmsec5}
    Let $I$ and $J$ be proper ideals in $V$ and $B$, respectively. Let $K=V \otimes^h J+ I \otimes^h B$. Then $K$ is primal in $V \otimes^h B$ if and only if $I$ is primal in $V$ and $J$ is primal in $B$.

\end{thm}

\begin{proof}

Let \( K = V \otimes^h J + I \otimes^h B \), and  let $I$ be primal in $V$ and $J$ be primal in $B$.  By Proposition \ref{result1} and Lemma 4.3,
\[
K= \bigcap \{ V\otimes^{h} J_0 + I_0 \otimes^{h} B : I_0 \in hull(I), J_0\in hull(J)\} .
\]

Since \( I \) and \( J \) are primal, by Proposition \ref{prop34}, there exist nets \( (P_\alpha)_{\alpha \in D} \subseteq \mathrm{Prim}(V) \) and \( (Q_\beta)_{\beta \in E} \subseteq \mathrm{Prim}(B) \) such that $P_\alpha \to P \quad \text{for each } P \in hull(I)$, and  $Q_\beta \to Q \quad \text{for each } Q \in hull(J)$. Define a directed set \( D \times E \) with the product order $(\alpha, \beta) \geq (\alpha', \beta') \iff \alpha \geq \alpha' \text{ and } \beta \geq \beta'$.
For each \( (\alpha, \beta) \in D \times E \), set
\[
R_{\alpha,\beta} = V \otimes^h Q_\beta + P_\alpha \otimes^h B.
\]
Then, by Theorem \ref{mainthm1}, we have \( R_{\alpha,\beta} \to V \otimes^h Q + P \otimes^h B \) for every \(P \in hull(I),   Q \in hull(J)\).

Now, let \( J_1, \ldots, J_{2n +1}\) be closed ideals in \( V \otimes^h B \) such that \( J_i \nsubseteq K \) for all \( i = 1, \ldots, 2n+1 \). For each \( i \), there exists \( I_0^i   \in hull(I), J_0^i\in hull(J)\) such that \( J_i \nsubseteq V \otimes^h J_0^i + I_0^i \otimes^h B \). Set $R_i= V \otimes^h J_0^i + I_0^i \otimes^h B$, for each $i=1,2,..., 2n+1$. Define the open sets $U(J_i) = \{ R \in \mathrm{Prim}(A \otimes_h B) : J_i \nsubseteq R \}$. Then \( U(J_i) \) is an open neighborhood of \( R_i \), and since \( R_{\alpha,\beta} \to R_i \), we have \( R_{\alpha,\beta} \in U(J_i) \) eventually for each \( i \). Hence, there exists \( (\alpha, \beta) \) such that $J_i \nsubseteq R_{\alpha,\beta} \quad \text{for all } i$. As \( R_{\alpha,\beta} \) is a prime ideal by (\cite{AA}, Theorem 3), it follows that $R_{\alpha,\beta} \nsupseteq J_1 J_2 \cdots J_{2n+1}$,so
$J_1 J_2 \cdots J_{2n+1} \neq \{0\}$. Therefore, \( K \) is primal.

Conversely, suppose \( K\) is primal. We will only show that $I$ is primal in $V$. Let $I_1$, $I_2$,...,$I_{2n+1}$ be ideals in $V$ with zero product then $$(I_1 \otimes^h B) (I_2 \otimes^h B)... (I_{2n+1} \otimes^h B)= (I_1I_2...I_{2n+1}) \otimes^h B= \{0\}$$ Thus, $I_i \otimes^h B \subseteq K$ for at least one value of $i$ which gives $I_i \subseteq I$. Thus, $I$ is primal.

\end{proof}

The next corollary is an immediate consequence of (\cite{AA}, Theorem $4(b)$) and Theorem \ref{mainthmsec5}.

\begin{cor}
   Every primitive ideal of $V \otimes^h B$ is primal. 
\end{cor}

Let \( I_1 \) and \( I_2 \) be ideals of a TRO \( V \). For \( J = I_1 \cap I_2 \), let \( \pi: V \to V/J \) be the natural quotient map. According to (\cite{ERR}, Proposition 3.1), the map \( \pi \otimes 1: V \otimes^h W \to V/J \otimes^h W \) is a quotient map. Consequently, \( \pi \otimes 1 (I_i \otimes^h W) \) is closed in \( V/J \otimes^h W \) and equal to \( I_i/J \otimes^h W \) for \( i = 1, 2 \) (\cite{AS}, Corollary 2.7). Utilizing this observation, we establish the following result.

\begin{prop} \label{result6}
Let \( V \) and \( W \) be TROs, and let \( \{I_i\}_{i=1}^n \) be a finite collection of ideals in \( V \). Then 
\[ \left( \sum_{i=1}^n I_i \right) \otimes^h W = \sum_{i=1}^n (I_i \otimes^h W). \]
\end{prop}

\begin{proof}
We shall only prove the result for \( n=2 \). Let \( I_1 \) and \( I_2 \) be two ideals in \( V \). Since \( I_1 + I_2 \) is closed (\cite{AA}, Lemma 2), we assume that \( V = I_1 + I_2 \). We shall proceed in two cases.

In the first case, assume that \( J = I_1 \cap I_2 = 0 \), so \( V \cong I_1 \oplus I_2 \). Clearly,
\[ I_1 \otimes^h W + I_2 \otimes^h W \subseteq V \otimes^h W. \]
For the reverse inclusion, consider the onto TRO homomorphisms \( P_i: V \to I_i \) defined as \( P_i(x_1, x_2) = x_i \) for \( i = 1, 2 \). Being TRO homomorphisms, each \( P_i \) is completely contractive, so \( P_i \otimes^h 1 : V \otimes^h W \to I_i \otimes^h W \) is completely contractive. Moreover, \( P_1 + P_2 = 1 \), thus for any \( u \in V \otimes^h W \), we have 
\[ u = (P_1 \otimes 1)(u) + (P_2 \otimes 1)(u) \in I_1 \otimes^h W + I_2 \otimes^h W. \]
Thus, the reverse inclusion is also established.

For the general case, note that \( \pi(I_1) \cap \pi(I_2) = 0 \), so applying the first case to \( V/J \otimes^h W \), we obtain 
\[ (V/J) \otimes^h W = (I_1/J) \otimes^h W + (I_2/J) \otimes^h W. \]
Therefore, \( \pi \otimes 1 \) maps \( I_1 \otimes^h W + I_2 \otimes^h W \) onto \( V/J \otimes^h W \). Moreover, \( I_1 \otimes^h W + I_2 \otimes^h W \) contains the kernel of \( \pi \otimes 1 \), which is \( J \otimes^h W \). Thus, by (\cite{AS}, Corollary 2.7),
\[ (\pi \otimes 1)^{-1}(\pi \otimes 1(I_1 \otimes^h W + I_2 \otimes^h W)) = I_1 \otimes^h W + I_2 \otimes^h W. \]
From this, we obtain 
\[ (I_1 + I_2) \otimes^h W = I_1 \otimes^h W + I_2 \otimes^h W. \]
\end{proof}

\begin{cor} \label{sec5cor}
    Let $P$ be an ideal of $V \otimes^h B$. Then $P$ is primal if and only if there are  primal ideals $I$ and $J$ of $V$ and $B$ respectively such that $P \supseteq I \otimes^h B + V \otimes^h J$. 
\end{cor}

\begin{proof}
    Let $K= I \otimes^h B + V \otimes^h J$. By last theorem, $K$ is primal and hence so is $P$.

    Conversely, let $P$ be a primal. By Zorn's lemma, we choose ideals $I$ and $J$ of $V$ and $B$, respectively, which are maximal with respect to the property that $V \otimes^h J \subseteq P$ and $I \otimes^h B \subseteq P$. Suppose $I_1, I_2,...I_{2n+1}$ are  ideals in $V$ with zero product. Since $P$ is primal and the  ideals $I_i \otimes^h B$ have zero product, $I_i \otimes^h B \subseteq P $ for at least one value of $i$. Hence by Proposition \ref{result6}, $$(I+I_i)\otimes^h B = I \otimes^h B+ I_i \otimes^h B  .$$ By the maximality of $I$, this implies that $I_i \subseteq I$. Thus, $I$ is primal. The same happens with $J$.
\end{proof}

In view of Theorem 5.2, the map $\Phi$ restricts to a map 
\[
\operatorname{Primal}(V) \times \operatorname{Primal}(B) \to \operatorname{Primal}(V \otimes^h B)
\]

In the following theorem, we establish that when restricted to minimal primal ideals, this map becomes a homeomorphism.

\begin{thm} \label{miniprimal}
    Let $P$ be a proper ideal in $V \otimes^h B$. Then $P$ is a closed minimal primal ideal of $V \otimes^h B$ if and only if $P = I \otimes^h B + V \otimes^h J$ for some minimal closed primal ideals $I$ and $J$ of $V$ and $B$, respectively. The map $(I, J) \mapsto V \otimes^h J + I \otimes^h B$ is a homeomorphism from $\operatorname{Min-Primal}(V) \times \operatorname{Min-Primal}(B)$ to $\operatorname{Min-Primal}(V \otimes^h B)$.
\end{thm}

\begin{proof}
    Suppose $I$ and $J$ are minimal  primal ideals of $V$ and $B$, respectively, and let $P = I \otimes^h B + V \otimes^h J$. Suppose $L \subseteq P$ for some primal ideal $L$ of $V \otimes^h B$. Thanks to Corollary \ref{sec5cor} there exist primal ideals $M$ of $V$ and $N$ of $B$ such that $L \supseteq V \otimes^h N + M \otimes^h B$, so 
\[
V \otimes^h N + M \otimes^h B \subseteq I \otimes^h B + V \otimes^h J.
\]
Thus, $M \subseteq I$ and $N \subseteq J$. Since $I$ and $J$ are minimal primal ideals, it follows that $I = M$ and $J = N$, which implies $L = P$. This shows that we have a canonical map \[ \operatorname{Min-Primal} (V) \times \operatorname{Min-Primal} (B) \to \operatorname{Min-Primal} (V \otimes h B).
\]
Let $P \in \operatorname{Min-Primal}(V \otimes^h B)$. Then, again in view of Corollary \ref{sec5cor}, $P \supseteq I \otimes^h B + V \otimes^h J$ for some closed primal ideals $I$ of $V$ and $J$ of $B$. By Theorem 5.1, $I \otimes^h B + V \otimes^h J$ is primal and so $P=I \otimes^h B + V \otimes^h J$  We claim that $I \in \operatorname{Min-Primal}(V)$. Suppose $M \subseteq I$ for some $M \in \operatorname{Primal}(V)$. Then 
\[
M \otimes^h B + V \otimes^h J \subseteq I \otimes^h B + V \otimes^h J = P.
\]
But since $P$ is minimal, we have $M \otimes^h B + V \otimes^h J = I \otimes^h B + V \otimes^h J$, and therefore $I = M$. Similarly, $J \in \operatorname{Min-Primal}(B)$. 
\end{proof}

    \begin{ex}
    Let \( \mathcal{H} \)  \( \mathcal{K} \) and \(\mathcal{L} \) be infinite-dimensional Hilbert spaces, and let \( u \in B(\mathcal{H}, \mathcal{K}) \) be a fixed partial isometry. Define \( V  \) to be the set of all sequences \( x = (x_n)_{n \geq 1} \) such that there exists a scalar \( \lambda(x) \in \mathbb{C} \) and a compact operator \( c(x) \in B(\mathcal{H}, \mathcal{K}) \) satisfying:
\[
\lim_{n \to \infty} \|x_{2n} - \lambda(x) u\| = 0 \quad \text{and} \quad \lim_{n \to \infty} \|x_{2n+1} - (\lambda(x) u + c(x))\| = 0.
\]
 With pointwise operations and the supremum norm, \( V \) is a TRO. For each \( n \geq 1 \), define the ideal $P_n = \{ x \in V : x_n = 0 \}$, which is a primitive ideal of \( V \). In fact, each \( P_n \) is minimal primal. To see this, suppose \( I \) is an ideal properly contained in \( P_n \). Let \( y \in P_n \setminus I \), and let \( J_1 \) be the ideal of \( V \) generated by \( y \). Set \( J_2 = J_1 \), and define $J_3 = \{ x \in V : x_k = 0 \text{ for all } k \ne n \}$. Then \( J_1 J_2^*J_3 = \{0\} \), but \( J_1 \not\subseteq I \) and \( J_2 \not\subseteq I \), so \( I \) is not primal. Hence \( P_n \) is minimal primal. Let $B= B(\mathcal{L})$, then $P_n \otimes^h B(\mathcal{L})+ V \otimes^h K(\mathcal{L}) $ is a minimal primal ideal of $V \otimes^h B(\mathcal{L})$ for each $n$. 
\end{ex}

 Our next aim is to provide a comprehensive classification of factorial ideals in \( V \otimes^h B \). Factorial ideals in Haagerup tensor product of \( C^* \)-algebras have been extensively studied (see, for example, \cite{AR}). We extend the concept to the Haagerup tensor product of TROs, providing new insights and classifications in this more general context.

\begin{defn}
     A representation $\pi$ of \( V \otimes^h B \) is called \emph{factorial} if the corresponding representation $\rho$ of \( V \otimes^{\text{tmax}} B \) is factorial.
\end{defn}

\begin{defn}
    A closed subspace $P$ of $V \otimes^h B$ is called a \emph{factorial} ideal of $V \otimes^h B$ if $P= \operatorname{ker}(\pi),$ for some  factorial representation $\pi$ of $V \otimes^h B$.
\end{defn}

\begin{prop}\label{facmain3}
    If $I$ and $J$ are factorial ideals of $V$ and $B$ respectively then $I \otimes^h B+ V \otimes^h J$ is a factorial ideal of $V \otimes^h B$.
\end{prop}
    
    \begin{proof}
     Let $I$ and $J$ be factorial ideals of $V$ and $B$ then by Proposition \ref{facmain}, $\mathcal{A}(I)$ is factorial ideal of $\mathcal{A}(V)$. Thus $\tilde{P}=\mathcal{A}(I) \otimes^h B+ \mathcal{A}(V) \otimes^h J$ is factorial ideal of $\mathcal{A}(V) \otimes^h B$ (\cite{AR}, Proposition $4.2$) and therefore $\tilde{P}= \operatorname{ker}(\psi)$ for some  factorial representation $\psi$ of $\mathcal{A}(V) \otimes^h B$. Let $\psi= \mathcal{A}(\rho)  \tilde{\epsilon}'$, where $\rho$ is a factor representation of $V \otimes^{\text{tmax}}B$. Define, $\pi= \rho  \epsilon'$, then $\pi$ is a factor representation of $V \otimes^h B$. We claim that $\operatorname{ker}(\pi)= I \otimes^h B+ V \otimes^h J$. Suppose $a \otimes b \in I \otimes B$, then  \[ 
\begin{bmatrix}
    0      & a \\
    0      & 0
\end{bmatrix}\otimes b  \in \mathcal{A}(I) \otimes B \subseteq \operatorname{ker}(\psi) =\operatorname{ker}(\mathcal{A}(\rho) \tilde{\epsilon}').
\] 
Thus, $\rho  \epsilon(a \otimes b)=0$, so $a \otimes b \in \operatorname{ker}(\pi)$. Since $\operatorname{ker}(\pi)$ is closed, so $I \otimes^h B \subseteq \operatorname{ker}(\pi)$. Similarly, $V \otimes^h J \subseteq \operatorname{ker}(\pi)$. Thus, $V \otimes^h J+ I \otimes^h B \subseteq \operatorname{ker}(\pi)$. Conversely, let $x \in \operatorname{ker}(\pi)$. Note that the following diagram

\[
\begin{tikzcd}
  V \otimes^h B
    \arrow[r, "\epsilon'"]
    \arrow[d, swap, "i"]
  &
  V \otimes^{\text{tmax}} B
    \arrow[d, swap, "j'"]
  \\
  \mathcal{A}(V) \otimes^h B
    \arrow[r, swap, "\tilde{\epsilon'}"]
  &
  \mathcal{A}(V)\otimes^{\mathrm{max}} B
\end{tikzcd}
\]

is commutative i.e. $j'  \epsilon'= \tilde{\epsilon}'i$. Using this, note that

  \[ \psi(i(x))=\mathcal{A}(\rho)  (j'(\epsilon'(x))=\mathcal{A}(\rho)
\Big (\begin{bmatrix}
    0      & \epsilon'(x) \\
    0      & 0
\end{bmatrix} \Big) = 
\begin{bmatrix}
         0    & \rho(\epsilon'(x)) \\
    0     & 0  
\end{bmatrix}=
\begin{bmatrix}
         0    & 0 \\
    0     & 0  
\end{bmatrix}.
\]

Therefore,  $i(x) \in \operatorname{ker}(\psi)=\mathcal{A}(I) \otimes^h B+ \mathcal{A}(V) \otimes^h J$ and therefore $x \in i^{-1}(\mathcal{A}(I) \otimes^h B+ \mathcal{A}(V) \otimes^h J)= I \otimes^h B+ V \otimes^h J$.
\end{proof}

\begin{prop}\label{facmain1}
     If $P$ is a factorial ideal of $V \otimes^h B$ then there exists factorial ideals $I$ and $J$ of $V$ and $B$ such that $P=I \otimes^h B+ M \otimes^h J$.
\end{prop}

\begin{proof}
    Let $P$ be a factorial ideal of $V \otimes^h B$, so $P= \operatorname{ker}(\pi)$ where $\pi=\rho \epsilon'$ and $\rho$ is a factorial representation of $V \otimes^{\text{tmax}} B$. Define $\psi= \mathcal{A}(\rho) \tilde{\epsilon'}$, then  $\operatorname{ker}(\psi)$ is a factorial ideal of $\mathcal{A}(V) \otimes^h B$. Thus, by (\cite{AR}, Proposition $4.2$) and Proposition \ref{facmain}, $\operatorname{ker}(\psi)= \mathcal{A}(I) \otimes^h B + \mathcal{A}(V) \otimes^h J$ for some factorial ideals $I$ and $J$. By the same argument as in Proposition \ref{facmain3}, it is not difficult to see that $P=\operatorname{ker}(\pi)= I \otimes^h B + V \otimes^h J$.
\end{proof}

The collection of all factorial ideals in $V \otimes^h B$ is denoted by $Fac(V \otimes^h B)$, and is endowed with the $\tau_\omega$-topology, which  it inherites from 
$\operatorname{Id}(V \otimes^h B)$ by the last proposition. From the above and the Theorem \ref{mainresult2}, we obtain the following result.

\begin{thm}
    Let $V$ be a TRO and $B$ be a $C^{\ast}$-algebra then the map Fac($V$) $\times$ Fac($B$) $\to$ Fac($V \otimes^h B$) defined by $(I, J) \to I \otimes^h B + V \otimes^h J$ is a homeomorphism. 
\end{thm}

\begin{rem}
    Since the kernels of factorial representations of a TRO $V$ are prime ideals,  every factorial ideal of $V \otimes^h B$ is necessarily prime.
\end{rem}

 \section{Glimm Ideals of $V \otimes^h B$}

In this section, we establish a complete characterization of \(\operatorname{Glimm}(V \otimes^h B)\) in terms of \(\operatorname{Glimm}(V)\) and \(\operatorname{Glimm}(B)\). We begin with the definition.

\begin{defn}
 We define $\operatorname{Glimm}(V \otimes^h B)$ as the complete regularization $\rho \operatorname{Prim}(V \otimes^h B)$ of $\ \operatorname{Prim}(V\otimes^h B)$, and denote by $\rho_{V\otimes^h B}: \operatorname{Prim}(V \otimes^h B) \to \operatorname{Glimm}(V \otimes^h B)$ the complete regularization map. There is  a one-to-one correspondence between the quotient space \(\operatorname{Prim}(V \otimes^h B) / \approx\) and a set of ideals in \(V \otimes^h B\), given by 
\[
[P] \mapsto \operatorname{k}([P]) = \bigcap \{ Q \in \operatorname{Prim}(V \otimes^h B) : Q \approx P \}.
\]
The ideals obtained through this correspondence are called the Glimm ideals of $V \otimes^h B$.  We will regard elements of $\operatorname{Glimm}(V \otimes^h B)$ as either points of a topological space or as ideals of $V \otimes^h B$, depending on the context.

\end{defn}

\begin{prop}\label{glimmmain1}
    If $I$ and $J$ are Glimm ideals of $V$ and $B$ respectively, then $I \otimes^h B+ V \otimes^h J$ is a Glimm ideal of $V \otimes^h B$.
\end{prop}

\begin{proof}
Let $I$ and $J$ be Glimm ideals of $V$ and $B$ then by  Proposition \ref{mainglimm1}, $\mathcal{A}(I)$
is Glimm ideal of $\mathcal{A}(V)$. Thus  $G=\mathcal{A}(I)\otimes^h B+ \mathcal{A}(V) \otimes^h J$ is Glimm ideal of  $\mathcal{A}(V) \otimes^h B$ (\cite{AR}, Theorem 4.8) and therefore   $G=q_{V \otimes^h B}(\phi_0(\mathcal{A}(P),Q))= k[\mathcal{A}(P)\otimes^h B+ \mathcal{A}(V) \otimes^h Q]$, $P\in  \operatorname{Prim}(V), Q \in \operatorname{Prim}(B)$, by (\cite{AR}, Lemma 4.6).

Using (\cite{AR}, Lemma 4.7), \[G=\bigcap_{
P \approx P', Q\approx Q', P'\in \operatorname{Prim}(V), Q'\in \operatorname{Prim}(B)} \mathcal{A}(V) \otimes^h Q' + \mathcal{A}(P') \otimes^h B.\]

Thus by (\cite{AA}, Lemma 4(c)),
\[
V \otimes^h J + I \otimes^h B = i^{-1}(G)
=\bigcap_{
P \approx P', Q\approx Q'} V \otimes^h Q' + P' \otimes^h B,
\]

and hence the result follows from (\cite{AA}, Theorem 4).

\end{proof}
\begin{prop}\label{glimmmmain1}
     If $G$ is a Glimm ideal of $V \otimes^h B$ then there exist Glimm ideals $I_G$ and $J_G$ of $V$ and $B$ such that $G=I_G \otimes^h B+ V \otimes^h J_G$.
\end{prop}

\begin{proof}
    Let $G$ be a Glimm ideal of $V \otimes^h B$, so $G= k([K])$ where $K\in \operatorname{Prim}(V\otimes^h B)$.  So there are closed prime ideals $P$ and $Q$ such that $K=P \otimes^h B+ V \otimes^h Q$ by (\cite{AA}, Theorem 4(b)). Choose $I\in hull(P), J\in hull(Q) $. Then 
    $ \mathcal{A}(I)\in \operatorname{Prim}( \mathcal{A}(V))$, $ \mathcal{A}(P) \otimes^h B+  \mathcal{A}(V) \otimes^h Q \subseteq  \mathcal{A}(I) \otimes^h B+  \mathcal{A}(V) \otimes^h J$, and so  $ \mathcal{A}(I) \otimes^h B+  \mathcal{A}(V) \otimes^h J \approx  \mathcal{A}(P) \otimes^h B+  \mathcal{A}(V) \otimes^h Q$. Hence, $ k[\mathcal{A}(I) \otimes^h B+  \mathcal{A}(V) \otimes^h J] = k[ \mathcal{A}(P) \otimes^h B+  \mathcal{A}(V) \otimes^h Q]$. Thus, by (\cite{AR}, Lemma $4.7$), $ \mathcal{A}(k[I]) \otimes^h B+  \mathcal{A}(V) \otimes^h k[J] = k[ \mathcal{A}(P) \otimes^h B+  \mathcal{A}(V) \otimes^h Q]$. By the same argument as in Proposition $\ref{glimmmain1}$,  $G=k[I] \otimes^h B+  V \otimes^h k[J]$, and hence the result follows.  
\end{proof}

In order to establish the homeomorphism of the restricted map 
\[
\operatorname{Glimm}(V) \times \operatorname{Glimm}(B) \to \operatorname{Glimm}(V \otimes^h B),
\]
we derive a result that closely follows the approach of (\cite{Lazer}, Theorem $3.2$). The key observation that enables this similarity is the structure of ideals in $V \otimes^h B$. 

Let \( I = \epsilon^{-1}(J) \) be an ideal of \( V \otimes^h B \), where \( J \) is an ideal of \( V \otimes^{\text{tmin}} B \). Then, \( \mathcal{A}(J) \) is an ideal of \( \mathcal{A}(V) \otimes^{\text{min}} B \), and by (\cite{EB}, Lemma 2.13), there exist ideals  
\[
I_{\mathcal{A}(V)} = \mathcal{A}(I_V) = \{ a \in \mathcal{A}(V) \mid a \otimes B \subseteq \mathcal{A}(J) \}
\]
and  
\[
I_B = \{ b \in B \mid \mathcal{A}(V) \otimes b \subseteq \mathcal{A}(J) \}
\]
of \( \mathcal{A}(V) \) and \( B \), respectively. Using the injectivity of \( \epsilon \), it follows that  
\[
I_V = \{ v \in V \mid v \otimes B \subseteq I \}, \quad
I_B = \{ b \in B \mid V \otimes b \subseteq I \}.
\]

Furthermore, we claim that if $I$ is prime, then $I_V$ and $I_B$  are also prime. Indeed, suppose that for two ideals $I_1$ and $I_2$ of $V$, if $I_1 \cap I_2 \subseteq I_V$, then $$(I_1 \otimes^h B) \cap (I_2 \otimes^h B)=(I_1 \cap I_2) \otimes^h B \subseteq I.$$ Since $I$ is prime, it follows that   $I_i \otimes^h B \subseteq I$ for some. Consequently, $I_i \subseteq I_V$, proving the claim. We will utilize this in the proof of the following lemma.

\begin{lem}\label{glimm2}
Let $V$ be a TRO and $B$ be a $C^*$-algebra, and let 
\[
\Phi : \operatorname{Id}'(V) \times \operatorname{Id}'(B) \to \operatorname{Id}'(V\otimes^{h} B)
\]
be the canonical homeomorphism. Then, for any topological space $Y$ and any continuous function 
\[
f : \operatorname{Prim}(V) \times \operatorname{Prim}(B) \to Y,
\]
the function 
\[
f \circ \Phi^{-1} : \Phi(\operatorname{Prim}(V) \times \operatorname{Prim}(B)) \to Y
\]
admits a unique continuous extension to $\operatorname{Prim}(V\otimes^h B)$. 
\end{lem}
\begin{proof}
    By (\cite{Lazer}, Lemma 3.1), there exists a continuous extension  
\[
\tilde{f} : (\operatorname{Prim}(V) \times \operatorname{Prim}(B))^c \to Y
\]  
of \( f \), here $X^c$ is the compactification of  a topological space $X$.  Furthermore, by (\cite{kirch}, Proposition 7.9), there exists a homeomorphism  
\[
\nu: \operatorname{Prim}(V)^c \times \operatorname{Prim}(B)^c = \operatorname{Prime}(V) \times \operatorname{Prime}(B) \to (\operatorname{Prim}(V) \times \operatorname{Prim}(B))^c,
\]  
which acts as the identity on the copies of  $\operatorname{Prim}(V) \times \operatorname{Prim
}(B)$   contained within these spaces.  

Now, let  \(\Psi: \operatorname{Id}(V \otimes^h B) \to \operatorname{Id}(V) \times \operatorname{Id}( B)\) be  the map given by
\[
\Psi(I) := (I_{V}, I_B).
\]
Then,  $\Psi$ is a  continuous map as $\Psi^{-1}(U(I_1) \times U(I_2))=\{I \in \operatorname{Id}(V \otimes^h B) \mid \operatorname{hull} (I_{V}))\cap V(I_1)\neq \emptyset, \operatorname{hull} (I_{B}) \cap V(I_2)\neq \emptyset\}=\{I \in \operatorname{Id}(V \otimes^h B) \mid  I_{V} \not\supseteq I_1,   I_{B} \not\supseteq I_2\}=\{I \in \operatorname{Id}(V \otimes^h B) \mid I \supseteq  V \otimes^h  I_{B}  \not\supseteq  V \otimes^h  I_{2}, I \supseteq  I_V \otimes^h  B  \not\supseteq  I_1 \otimes^h  B\} = \{I \in \operatorname{Id}(V \otimes^h B) \mid I \not\supseteq  V \otimes^h  I_{2}, I   \not\supseteq  I_1 \otimes^h  B\}= U(V \otimes^h  I_{2}) \cap  U(I_1 \otimes^h  B) $, which is open in  $V \otimes^h B$. Furthermore, $\Psi$ maps $\operatorname{Prime}(V \otimes^h B)$ onto $\operatorname{Prime}(V) \times \operatorname{Prime}( B) $  by (\cite{AA}, Theorem $3(a)$).

The composition  
\[
\tilde{f} \circ \nu \circ \Psi : \operatorname{Prime}(V \otimes^h B) \to Y
\]  
is then continuous. Now, let  
\[
\hat{f} = (\tilde{f} \circ \nu \circ \Psi) |_{\operatorname{Prim}(V \otimes^h B)}.
\]  

Finally, for any \( (P_1, P_2) \in \operatorname{Prim}(V) \times \operatorname{Prim}( B) \), we have  
\[
\hat{f}(\Phi(P_1, P_2)) = f(\nu(P_1, P_2)) = f(P_1, P_2),
\]  
since \( \Psi(\Phi(P_1, P_2)) = (P_1, P_2) \). This confirms that \( \hat{f} \) extends \( f \) as desired.
\end{proof}

To establish the homeomorphism,  we equip \( \operatorname{Glimm}(V) \times \operatorname{Glimm}(B) \) with the topology $\tau_{cr}$, which arises from the complete regularization of \(\operatorname{Prim}(V) \times \operatorname{Prim}(B)\). The theorem below shows that the restriction of the map $\Phi$ to this space is a homeomorphism onto $\operatorname{Glimm}(V \otimes^h B)$.

\begin{thm}\label{sec6main}
    Let $V$ be a TRO and $B$ be a $C^{\ast}$-algebra. Then restriction of the map $\Phi$  $$(\operatorname{Glimm}(V) \times \operatorname{Glimm}(B), \tau_{cr}) \to \operatorname{Glimm}(V \otimes^h B)$$   is a homeomorphism. 
\end{thm}

\begin{proof}
The map  $$\rho_V \times \rho_B : \operatorname{Prim}(V) \times \operatorname{Prim}(B) \to (\operatorname{Glimm}(V) \times \operatorname{Glimm}(B), \tau_{cr})$$ is the complete regularization of \(\operatorname{Prim}(V) \times \operatorname{Prim}(B)\) by (\cite{Glimm}, Lemma $1.1$). By Lemma \ref{glimm2}, the composition  
\[
(\rho_V \times \rho_B) \circ \Phi^{-1} : \Phi(\operatorname{Prim}(V) \times \operatorname{Prim}(B)) \to \operatorname{Glimm}(V) \times \operatorname{Glimm}(B)
\]  
extends uniquely to a continuous map  
\[
\overline{(\rho_V \times \rho_B)} : \operatorname{Prim}(V\otimes^h B) \to \operatorname{Glimm}(V) \times \operatorname{Glimm}(B).
\]
Since \[\operatorname{Glimm}(V) \times \operatorname{Glimm}(B)\] is completely regular, this induces a continuous and surjective map \[ \psi : \operatorname{Glimm}(V\otimes^h B) \to \operatorname{Glimm}(V) \times \operatorname{Glimm}(B), \] that satisfies \[ \psi \circ \rho_{(V \otimes^h B)} = \overline{\rho_V \times \rho_B}, \]
by (\cite{walk}, Corollary 1.8).

To show that \(\psi\) is a homeomorphism, it suffices to prove that the induced \(*\)-homomorphism \[ \psi *: C^b(\operatorname{Glimm}(V) \times \operatorname{Glimm}(B)) \to C^b(\operatorname{Glimm}(V \otimes^h B)), \quad \psi^*(f) = f \circ \psi
\]  
is surjective (\cite{rings}, Theorem 10.3(b)). For this, let \(f \in C^b(\operatorname{Glimm}(V\otimes^h B))\). Then $f \circ \rho_{(V \otimes^h B)} \in C^b(\operatorname{Prim}(V\otimes^h B))$, which implies  $ f \circ \rho_{(V \otimes^h B)} \circ \Phi \in C^b(\operatorname{Prim}(V) \times \operatorname{Prim}(B))$. 
Let \(g \in C^b(\operatorname{Glimm}(V) \times \operatorname{Glimm}(B))\) be the unique function satisfying 
\[g \circ (\rho_V \times \rho_B) = f \circ \rho _{V \otimes h B} \circ \Phi.\] 
both \(f \circ \rho_{(V \otimes^h B)}\) and \(g \circ (\rho_V \times \rho_B)\) are continuous extensions of \(g \circ (\rho_V \times \rho_B) \circ \Phi^{-1}\) to \(\operatorname{Prim}(V\otimes^h B)\), they must be equal by the lemma \ref{glimm2}.  

Now, for any \(m \in \operatorname{Glimm}(V\otimes^h B)\), choose \(M \in \operatorname{Prim}(V\otimes^h B)\) such that \(\rho_{(V \otimes^h B)}(M) = m\). Then,  
\[
\begin{aligned}
\psi^*(g)(m) &= (g \circ \psi)(m) \\
&= (g \circ \psi \circ \rho_{(V \otimes^h B)})(M) \\
&= (g \circ (\rho_V \times \rho_B))(M) \\
&= (f \circ \rho_{(V \otimes^h B)})(M) \\
&= f(m).
\end{aligned}
\]
Thus, \(\psi^*(g) = f\), proving that \(\psi^*\) is surjective. This concludes the proof that \(\psi\) is a homeomorphism.

\end{proof}

The following corollary follows immediately from Corollary \ref{sector3main} and Theorem \ref{sec6main}.

\begin{cor}
    Suppose $V$ or $B$ satisfies any of the conditions $(i)-(iii)$ in Corollary \ref{sector3main}. Then $\tau_{cr}= \tau_p$ on $\operatorname{Glimm}(V) \times \operatorname{Glimm}(B)$,  $\operatorname{Glimm}(V \otimes^h B)$ is homeomorphic to $(\operatorname{Glimm}(V) \times \operatorname{Glimm}(B), \tau_p)$. In particular, if $\rho_V$ or $\rho_B$ is open, this homeomorphism holds. 

\end{cor}

Similarly, the next corollary follows immediately from Proposition \ref{mainglimm1} and Theorem \ref{sec6main}.
\begin{cor}
    Suppose that both $\rho_V$ and $\rho_B$ are open. Then $\tau_p= \tau_q$ on \\ $\operatorname{Glimm}(V) \times \operatorname{Glimm}(B)$, and hence $\operatorname{Glimm}(V \otimes^h B)$ is homeomorphic to \\ $(\operatorname{Glimm}(V) \times \operatorname{Glimm}(B), \tau_q)$. 
\end{cor}

We now turn to certain natural consequences that follow from the theory developed so far concerning the structure of primal and Glimm ideals in \( V \otimes^h B \). In particular, we observe a structural relationship between these two classes of ideals. It is a well-established fact in the theory of \( C^* \)-algebras that every primal ideal contains a Glimm ideal. Within the framework established above, this result continues to hold in our setting. Furthermore, we show that if every Glimm ideal is primal, then each Glimm ideal is necessarily minimal among the primal ideals. 

\begin{cor}
Every proper primal ideal $J$ of \( V \otimes^h B \) contains a Glimm ideal of $V \otimes^h B.$
\end{cor}

\begin{proof}
    By Corollary \ref{sec5cor}, there exist primal ideals \( P \subseteq V \) and \( Q \subseteq B \) such that $ J \supseteq P \otimes^h B + V \otimes^h Q$. By (\cite{qs1991}, Lemma $2.2$) and Proposition \ref{glimmmain1},   there exist Glimm ideals \( I_\alpha \subseteq V \) and \( J_\alpha \subseteq B \) with \( P \supseteq I_\alpha \) and \( Q \supseteq J_\alpha \). It follows that $J \supseteq I_\alpha \otimes^h B + V \otimes^h J_\alpha,$ where the ideal \( I_\alpha \otimes^h B + V \otimes^h J_\alpha \) is a Glimm ideal of \( V \otimes^h B \) by Proposition $6.4$.

\end{proof}

\begin{cor}
 If every Glimm ideal of $V \otimes^h B$ is primal, then $$\mathrm{Glimm}(V \otimes^h B) = \mathrm{Min}\text{-}\mathrm{Primal}(V \otimes^h B)$$
as sets.
\end{cor}

\begin{proof}
    In view of Theorem \ref{miniprimal} and Theorem \ref{sec6main}, it is enough to prove that $\mathrm{Glimm}(V) = \mathrm{Min}\text{-}\mathrm{Primal}(V)$ and $\mathrm{Glimm}(B) = \mathrm{Min}\text{-}\mathrm{Primal}(B)$. This follows immediately from  (\cite{qs1991}, Lemma $3.1$), along with Proposition \ref{primalmain} and Proposition \ref{glimmmain1}. 
\end{proof}

As one more application of the theory developed above concerning the structure of primal and Glimm ideals in $V \otimes^h B$, we now turn to the study of quasi-standardness of $V \otimes^h B$. Following the approach used for TROs, we define an equivalence relation \(\sim\) on $\operatorname{Prim}(V \otimes^h B)$ based on topological separation, exactly in the same manner as done earlier. We then say that $V \otimes^h B$ is quasi-standard if the equivalence relation \(\sim\) is open.

\begin{cor}
   The Haagerup tensor product $V \otimes^h B $ is quasi-standard if and only if $V$ and $B$ are quasi-standard.
\end{cor}

This result is known for $C^{\ast}$-algebras (\cite{qs1991}, Theorem $5.4$). In view of the theory developed so far and the openness of the equivalence relation \(\sim\)—an analogous proof applies to the TRO setting as well. We omit the details.

We conclude this section by summarizing our results in the following example, which illustrates the characterization of ideals in $B(\mathcal{H}, \mathcal{K}) \otimes^h B(\mathcal{L})$.

\begin{ex}
Let $\mathcal{H}$, $\mathcal{K}$, and $\mathcal{L}$ be infinite-dimensional separable Hilbert spaces. It is known (\cite{AKKKK}, Example 4.23) that any nonzero ideal of $B(\mathcal{H}, \mathcal{K}) \otimes^h B(\mathcal{L})$ is one of the following:
$B(\mathcal{H}, \mathcal{K}) \otimes^h B(\mathcal{L}), B(\mathcal{H}, \mathcal{K}) \otimes^h K(\mathcal{L}), K(\mathcal{H}, \mathcal{K}) \otimes^h B(\mathcal{L}), K(\mathcal{H}, \mathcal{K}) \otimes^h K(\mathcal{L}),B(\mathcal{H}, \mathcal{K}) \otimes^h K(\mathcal{L}) + K(\mathcal{H}, \mathcal{K}) \otimes^h B(\mathcal{L})$. From Corollary \ref{sec5cor}, we deduce that every ideal of $B(\mathcal{H}, \mathcal{K}) \otimes^h B(\mathcal{L})$ is primal. It is straightforward to verify that $K(\mathcal{H}, \mathcal{K}) \otimes^h K(\mathcal{L})$ is not a prime ideal of $B(\mathcal{H}, \mathcal{K}) \otimes^h B(\mathcal{L})$, and hence, it is not factorial (Remark 5.12). Furthermore, by Proposition \ref{facmain1}, all other nontrivial ideals of $B(\mathcal{H}, \mathcal{K}) \otimes^h B(\mathcal{L})$ are factorial ideals. It is known (\cite{AA}, Example 4) that the only primitive ideal of $B(\mathcal{H}, \mathcal{K}) \otimes^h B(\mathcal{L})$ is $B(\mathcal{H}, \mathcal{K}) \otimes^h K(\mathcal{L}) + K(\mathcal{H}, \mathcal{K}) \otimes^h B(\mathcal{L})$. Thus, the only Glimm ideal of $B(\mathcal{H}, \mathcal{K}) \otimes^h B(\mathcal{L})$ is also $B(\mathcal{H}, \mathcal{K}) \otimes^h K(\mathcal{L}) + K(\mathcal{H}, \mathcal{K}) \otimes^h B(\mathcal{L})$. Finally, since $B(\mathcal{H}, \mathcal{K})$ and $B(\mathcal{L})$ are quasi-standard, it follows that $B(\mathcal{H}, \mathcal{K}) \otimes^h B(\mathcal{L})$ is also quasi-standard.

\end{ex}

% ------------------------------------------------------------------------
\end{document}